\documentclass[11pt]{article}

\usepackage[a4paper, left=3.2cm,top=3cm,right=3.2cm,bottom=3.8cm]{geometry}
\usepackage{concmath}
\usepackage[T1]{fontenc}
\usepackage[utf8]{inputenc}

\usepackage{authblk}
\usepackage{amsmath,amsfonts,amssymb,amsthm}
\usepackage{graphicx}
\graphicspath{{pics/},{dnett/}}
\usepackage[colorlinks=true]{hyperref}
\hypersetup{urlcolor=blue, citecolor=red}
\usepackage{mathtools}
\usepackage{framed}
\usepackage{enumitem}
\usepackage{dsfont}

\usepackage{siunitx}
\usepackage{subcaption}
\usepackage{xcolor}
\usepackage{soul}

\usepackage[boxruled,noline]{algorithm2e}

\title{Discretization of learned NETT regularization for  solving inverse problems}
\date{}

\author{Stephan Antholzer}

\affil{Department of Mathematics, University of Innsbruck\authorcr
Technikerstrasse 13, 6020 Innsbruck, Austria\authorcr
 {\tt stephan.antholzer@uibk.ac.at}
 }

\author{Markus Haltmeier}

\affil{Department of Mathematics, University of Innsbruck\authorcr
Technikerstrasse 13, 6020 Innsbruck, Austria\authorcr
 {\tt markus.haltmeier@uibk.ac.at}
 }

\newtheorem{theorem}{Theorem}[section]
\newtheorem{corollary}[theorem]{Corollary}

\newtheorem{lemma}[theorem]{Lemma}
\newtheorem{proposition}[theorem]{Proposition}

\theoremstyle{definition}

\newtheorem{remark}[theorem]{Remark}
\newtheorem{assumption}[theorem]{Assumptions}

\newcommand*{\N}{\mathbb{N}}

\newcommand*{\R}{\mathbb{R}}

\newcommand{\Bo}{\mathbf{B}}

\newcommand{\Ao}{\mathbf{A}}
\newcommand{\Wo}{\mathcal{W}}
\newcommand{\Wd}{\mathbf{W}}
\newcommand{\Mo}{\mathbf{M}}
\newcommand{\Uo}{\mathbf{U}}
\newcommand{\Vo}{\mathbf{V}}
\newcommand{\Sigmao}{\boldsymbol{\Sigma}}

\newcommand*{\transpose}[1]{{#1}^\intercal}

\newcommand{\X}{\mathbb{X}}
\newcommand{\dom}{\mathbb{D}}
\newcommand{\Y}{\mathbb{Y}}
\newcommand{\net}{\boldsymbol \Phi}

\newcommand*{\reg}{\mathcal{R}}
\newcommand*{\tik}{\mathcal{T}}

\newcommand*{\dist}{\mathcal{D}}

\newcommand*{\Breg}{\mathcal{B}}

\newcommand{\data}{y}
\newcommand{\signal}{x}
\newcommand{\rr}{\mathbf{r}}
\newcommand{\sr}{\mathbf{s}}

\newcommand{\bad}{h}
\newcommand{\zsignal}{z}

\renewcommand{\phi}{\varphi}

\DeclarePairedDelimiter{\abs}{\lvert}{\rvert}
\DeclarePairedDelimiter{\norm}{\lVert}{\rVert}
\DeclarePairedDelimiter{\innerprod}{\langle}{\rangle}

\DeclareMathOperator*{\argmin}{argmin}

\newcommand{\edot}{\,\cdot\,}

\def\plus{{\boldsymbol{\texttt{+}}}}

\definecolor{drot}{rgb}{0,0,0}
\definecolor{blor}{rgb}{1,0,1}
\definecolor{blgr}{rgb}{0,1,1}
\definecolor{bblau}{rgb}{0,0,1}
\definecolor{goyel}{rgb}{0.3,0,1}
\definecolor{orred}{rgb}{0,0.39,0}

\allowdisplaybreaks

\begin{document}
\maketitle

\begin{abstract}
Deep learning based reconstruction methods deliver outstanding results for solving inverse problems and are therefore becoming increasingly important. A recently invented class of learning-based reconstruction methods is the so-called NETT (for Network Tikhonov Regularization), which contains a trained neural network as regularizer in generalized Tikhonov regularization. The existing analysis of NETT considers fixed operator and fixed regularizer and analyzes the convergence as the noise level in the data approaches zero. In this paper, we extend the frameworks and analysis considerably to reflect various practical aspects and  take into account discretization of the data space, the solution space, the forward operator and the neural network defining the regularizer. We show the asymptotic convergence of the discretized NETT approach for decreasing noise levels and discretization errors.   Additionally, we derive convergence rates and present numerical results for a limited data problem in photoacoustic tomography.
\end{abstract}

\section{Introduction}

In this paper, we are interested in neural network based solution of inverse problems
of the form
\begin{equation} \label{eq:ip}
	\text{ Find  $\signal$ from data } \quad \data^\delta  = \Ao \signal +  \eta \,.
\end{equation}
Here $\Ao$ is a potentially non-linear operator between Banach spaces $\X$ and $\Y$, $\data^\delta$ are the given noisy data,   $\signal$ is the unknown to be recovered, $\eta$ is the unknown noise perturbation and $\delta \geq 0$ indicates the noise level. Numerous image reconstruction problems, parameter identification tasks or geophysical applications can be stated as such inverse problems  \cite{engl1996regularization,scherzer2009variational,natterer2001mathematical,zhdanov2002geophysical}. Special challenges  in solving inverse problems  are the non-uniqueness of the solutions and the instability of the solutions with respect to the given data.  To overcome  these issues, regularization methods are needed, which are used as criteria for selecting specific solutions and at the same time stabilize the inversion process.

\subsection*{Reconstruction with learned regularizers}

One of the most established class of methods for solving inverse problems is variational regularization  where regularized solutions are defined as minimizers of the  Tikhonov functional \cite{scherzer2009variational,morozov1984methods,tikhonov1977solution}
\begin{equation}\label{eq:tik}
    	 \tik_{\data^\delta,\alpha} \colon \X \to  [0, \infty]  \colon 
	 \signal \mapsto 
	 \dist(\Ao \signal, \data^\delta) + \alpha\reg(\signal) \,.
\end{equation}
Here $\dist$ is a distance like function measuring closeness of the data, $\reg$ a regularization term enforcing regularity of the minimizer and $\alpha$ is the regularization parameter.  In the case that $\dist$ and the regularizer are defined by the Hilbert space norms, \eqref{eq:tik}  is classical Tikhonov regularization for which the theory is quite complete
\cite{engl1996regularization,ivanov2002theory}.  In particular, in this case, convergence rates, which name  quantitative  estimates for the distance between the true  and regularized solutions  are well known. Convergence rates for non-convex regularizers are derived  in \cite{grasmair2010generalized}.

Typical regularization techniques are based on simple hand crafted regularization terms  such as  the total variation  $\norm{f}_{\rm TV} = \int  \abs{\nabla f}$ or quadratic  Sobolev norms $ \norm{\nabla f}_2^2 = \int  \abs{\nabla f}^2$  on some function space. However, these regularizers are quite simplistic and might not well reflect the actual complexity of the underlying class of functions. Therefore, recently, it has been proposed and analyzed in \cite{li2020nett}  to  use machine learning to construct regularizers in a data driven manner. The strategy in \cite{li2020nett}  is to construct a data-driven regularizer via the  following consecutive steps:
\begin{enumerate}[label=(T\arabic*), ref=T\arabic*]
\item \label{T1}
Choose a family  of desired reconstructions $(\signal_i)_{i =1}^n$.
\item \label{T2} For some $\Bo \colon \Y \to \X$, construct undesired reconstructions  $(\Bo\Ao\signal_i)_{i =1}^n$.

\item \label{T3}
Choose a class $(\net_\theta)_{\theta \in \Theta}$ of functions (networks) $\net_\theta \colon \X \to  \X $.

\item \label{T4}
Determine  $\theta^\star \in \Theta $ with $\net_{\theta^\star}(\signal_i)  \simeq \signal_i \wedge \net_{\theta^\star}(\Bo\Ao\signal_i)  \simeq \signal_i$.

\item \label{T5}
Define $\reg (\signal)  = r(\signal,  \net(\signal) ) $ with $\net  =  \net_{\theta^\star}$  for some $r \colon \Y \times \Y \to [0, \infty]$.
\end{enumerate}
For imaging applications, the function class $(\net_\theta)_{\theta \in \Theta}$ can be chosen as  convolutional neural networks which have demonstrated to give powerful classes of mappings between  image spaces. The function  $r$ measures distance between a potential reconstruction $\signal$ and the output of the network $\net(\signal)$, and possibly adds additional regularization  \cite{obmann2019sparse,obmann2021augmented}.  According to the training strategy in item (\ref{T4}) the value of the regularizer will be small if the reconstruction is similar to elements  in $(\signal_i)_{i=1}^n$ and large for elements in $(\Bo\Ao\signal_i)_{i=1}^n$.   A simple example that we will use for our numerical results is the learned regularizer  $\reg (\signal)  = \norm{\signal-  \net(\signal) }^2 + \norm{\signal}_{\rm TV}$.

Convergence analysis and convergence rates for  NETT as well as training strategies have been established in   \cite{haltmeier2020regularization,li2020nett,obmann2021augmented}. A different training strategy  for learning a regularizer has been proposed in \cite{lunz2018adversarial,mukherjee2020learned}.   Note that learning the regularizer first and then minimizing the Tikhonov functional is different from variational and iterative networks   \cite{adler2017solving,aggarwal2018modl,de2019deep,kobler2017variational,ADMMnet} where  an iterative scheme is applied to enroll  the functional $\dist_{ \theta} (\Ao \signal, \data^\delta) + \alpha \reg_ \theta(\signal ) $ which is then trained in an end to end fashion.
Training the regularizer first  has the advantage of being more modular, sharing some similarity with plug and play techniques \cite{romano2017little}, and the network training is independent of the forward operator $\Ao$. Moreover, it enables to derive a convergence analysis as the noise level tends to zero and therefore comes with theoretical recovery guarantees.

\subsection*{Discrete NETT}

The existing analysis of NETT  considers minimizers of the Tikhonov functional \eqref{eq:tik} with regularizer of the form $\reg (\signal)  = r(\signal,  \net(\signal) )$ before  discretization, typically in an infinite dimensional setting. However, in practice, only finite dimensionale approximations of the unknown, the operator and the neural network are given.   To address these issues, in this paper, we study   discrete NETT regularization which considers minimizers of
\begin{equation}\label{eq:NETTdiscrete}
 	\tik_{\data^{\delta},\alpha,n} \colon \X_n \to \Y
	\colon \signal \mapsto
	\dist(\Ao_n z, \data^{\delta}) + \alpha \reg_n(z)\,.
\end{equation}
Here $(\X_n)_{n\in \N}$, $(\Ao_n)_{n \in \N}$ and $(\reg_n)_{n \in\N}$ are  families of subspaces of $\X_n \subseteq \X$, mappings $\Ao_n \colon \X \to \Y$ and  regularizers  $\reg_n \colon \X \to [0,\infty]$, respectively, which reflect discretization of all involved operations.    We present  a full convergence  analysis as the noise level $\delta $ converges to zero and  $n, \alpha$ are chosen accordingly.   Discretization of variational regularization has studied in  \cite{poschl2010discretization} for the case that  $\dist$ is given by the norm distance and the  regularizer $\reg$ is taken convex and fixed. However, in the case of discrete NETT regularization it is natural to consider the case where the regularization depends on the discretization as regularization is learned in a discretized setting based on actual data.  For that purpose our analysis includes non-convex regularizers that are allowed to depend on the discretization and the noise level.

\subsection*{Outline}

The convergence analysis  including  convergence rates  is presented in Section~\ref{sec:NETT}. In Section~\ref{sec:PAT} we will  present numerical results for a non-standard  limited data  problem in photoacoustic tomography that can be considered as simultaneous inpainting and artifact removal problem. We conclude the paper with a short summary and conclusion presented in Section~\ref{sec:conclusion}.

\section{Convergence analysis}
\label{sec:NETT}

In this  section we study the  convergence of   \eqref{eq:NETTdiscrete}  and derive convergence rates.

\subsection{Well-posedness}

First we state the assumptions  that we will use for well-posedness (existence and stability of minimizing NETT).

\begin{assumption}[Conditions for well-posedness]\label{ass-well}\mbox{}
\begin{enumerate}[label=(W\arabic*), ref=W\arabic*]
    \item\label{a1} $\X$, $\Y$ are  Banach spaces, $\X$ reflexive,
    $\dom \subseteq \X$ weakly sequentially closed.

     \item\label{a-dist}  The distance measure $\dist\colon \Y \times \Y \to [0, \infty]$ satisfies
  \begin{enumerate}
           \item\label{a-dist1}
           $\exists \tau\geq1\colon \forall \data_1,\data_2,y_3\in \Y\colon \dist(\data_1,\data_2) \leq \tau \dist(\data_1,y_3) + \tau\dist(y_3,\data_2)$.

    \item\label{a-dist2}
    $\forall \data_1,\data_2\in \Y \colon \dist(\data_1,\data_2) = 0 \Leftrightarrow \data_1=\data_2$.

        \item\label{a-dist3}
         $\forall  \data, \tilde{\data}  \in \Y  \colon \dist(\data, \tilde\data) < \infty \wedge \norm{ \tilde \data - \data_k} \rightarrow 0  \Rightarrow \dist(\data, \data_k) \rightarrow \dist(\data, \tilde\data)$.

             \item\label{a-dist4}
    $\forall \data \in \Y  \colon
     \norm{\data_k - \data}  \to 0 \Rightarrow \dist(\data_k,\data)\to 0 $.

      \item\label{a-dist5}
       $\dist$ is weakly sequentially lower semi-continuous (wslsc).
   \end{enumerate}

   \item\label{a-reg}  $\reg\colon \X \to [0,\infty]$ is proper and wslsc.

      \item\label{a-a} $\Ao\colon \dom \subseteq \X \to \Y$ is  weakly sequentially continuous.

  \item \label{a-c}  $\forall \data , \alpha, C \colon \{\signal \in \X \mid \tik_{\data,\alpha} \leq C\}$ is  nonempty and  bounded.

        \item \label{a6}  $(\X_n)_{n\in\N}$ is a sequence of subspaces of $\X$.

    \item\label{a7}   $(\Ao_n)_{n\in\N}$
    is a family of  weakly sequentially continuous $\Ao_n \colon \dom\to \Y$.

    \item\label{a8}  $(\reg_n)_{n\in\N}$
    is a family of proper wslsc regularizers $\reg_n \colon \X \to [0,\infty]$.

     \item \label{a9}
    $ \forall \data, \alpha, C, n  \colon \{ \signal \in \X_n \mid \tik_{\data,\alpha,n} \leq C \}$ is nonempty and  bounded.
\end{enumerate}
\end{assumption}

Conditions  (\ref{a-dist})-(\ref{a-c}) are quite standard for Tikhonov regularization in Banach spaces to guarantee the existence and stability of minimizers of the Tikhonov functional and the given conditions are similar to   \cite{grasmair2010generalized,haltmeier2020regularization,li2020nett,obmann2019sparse,poschl2008tikhonov,scherzer2009variational,tikhonov1998nonlinear}. In particular,  (\ref{a-dist}) describes the properties that the distance measure $\dist$ should have.  Clearly, the norm distance on $\Y$ fulfills these properties.  Item (\ref{a-dist3}) is the continuity of $\dist(\data, \cdot)$ while (\ref{a-dist4})  considers the continuity of $\dist(\cdot, \data)$ at $\data$.  While (\ref{a-dist3}) is not needed for existence and convergence of NETT it is required for the stability result as shown in \cite[Example 2.7]{obmann2019sparse}. Assumption~(\ref{a-c}) is a coercivity condition; see \cite[Remark 2.4f.]{li2020nett} on how to achieve this for a regularizer defined by neural networks.  Note that for convergence and convergence rates we will require additional conditions that concern the discretization of the reconstruction space, the forward operator and regularizer.

The   references     \cite{grasmair2010generalized,li2020nett,obmann2019sparse,poschl2008tikhonov}  all consider  general distance measures and allow non-convex regularizers. However,  existence and stability of minimizing \eqref{eq:tik} are  shown under assumptions slightly different from  (\ref{a1})-(\ref{a-c}).  Below we therefore give a short proof of the existence and stability results.

\begin{theorem}[Existence and Stability]
    Let Assumption~\ref{ass-well} hold.  Then for all $\data^\delta\in \Y$, $\alpha>0$, $n\in\N$ the following assertions hold true:
    \begin{enumerate}[label=(\alph*), ref=\alph* ]
        \item \label{well-1} $ \argmin  \tik_{\data,\alpha,n} \neq \emptyset$.

            \item  \label{well-2}
            Let $(\data_k)_{k\in\N}\in \Y^\N$ with $\data_k \to \data$
            and consider $\signal_k \in \argmin \tik_{\data_k,\alpha,n}$.
            \begin{itemize}
            \item $(\signal_k)_{k\in\N}$ has at least one weak accumulation point.
            \item Every  weak accumulation point $(\signal_k)_{k\in\N}$ is a minimizer of
            $\tik_{\data,\alpha,n}$.
    \end{itemize}
           \item    \label{well-3}
                        The statements in  (\ref{well-1}),(\ref{well-2}) also hold for $\tik_{\data,\alpha}$ in place of  $\tik_{\data,\alpha,n}$,
    \end{enumerate}
\end{theorem}

\begin{proof}
Since (\ref{a1}), (\ref{a6})-(\ref{a9}) for  $\tik_{\data,\alpha,n}$ when $n \in \N$ are fixed give the same assumption as   (\ref{a1}), (\ref{a-reg})-(\ref{a-c}) for the non-discrete   counterpart $\tik_{\data,\alpha}$, it is sufficient to verify  (\ref{well-1}), (\ref{well-2})  for the latter. Existence of minimizers follows from  (\ref{a1}), (\ref{a-dist5}), (\ref{a-reg})-(\ref{a-c}), because these items imply that the $\tik_{\data,\alpha}$ is a wslsc coercive functional defined on a nonempty weakly sequentially closed subset of a reflexive Banach space. To show stability  one notes that  according to (\ref{a-dist1}) for all  $\signal \in \X$ we have
\begin{multline*}
\dist(\Ao \signal_k , \data) + \alpha \reg ( \signal_k)
\leq
\tau \bigl( \dist(\Ao \signal_k , \data_k) + \alpha \reg(\signal_k) \bigr) + \tau \dist(\data , \data_k)
\\\leq
\tau \bigl( \dist(\Ao \signal , \data_k) + \alpha \reg(\signal) \bigr) + \tau \dist(\data , \data_k) \,.
\end{multline*}
According to (\ref{a-dist3}), (\ref{a-dist4}), (\ref{a-c})  there exists $\signal \in \X$ such that the right hand side is bounded, which by  (\ref{a-c}) shows that $(\signal_k)_k$ has a weak accumulation point. Following the standard proof  \cite[Theorem 3.23]{scherzer2009variational} shows that weak accumulation points satisfy the claimed properties.
\end{proof}

In the following we write   $\signal_{\alpha,n}^{\delta}$ for minimizers of  $\tik_{\data^{\delta},\alpha,n}$. For $\data\in \Y$ we call  $\signal^\plus \in \argmin \{ \reg(\signal ) \mid \signal \in \X \wedge \Ao \signal =  \data  \}$ an $\reg$-minimizing solution of $\Ao \signal = \data$.

\begin{lemma}[Existence of $\reg$-minimizing solutions]
Let Assumption~\ref{ass-well} hold. For any $\data\in \Ao(\dom)$ an  $\reg$-minimizing solution of $\Ao \signal = \data$ exists. Likewise, if $n \in \N$ and $\data \in \Ao_n(\dom)$ an  $\reg_n$-minimizing solution of $\Ao_n \signal = \data$ exists.
\end{lemma}

\begin{proof}
Again is is sufficient the verify the claim for $\reg$-minimizing solution.
Because $\data\in \Ao(\dom)$, the set $\Ao^{-1}(\{ \data \} ) = \{ \signal  \in \X \mid \Ao \signal =  \data  \}$  is  non-empty. Hence we  can choose a sequence $(\signal_k)_{k \in \N} $ in $\Ao^{-1}(\{ \data \} )$ with $\reg(\signal_k)  \to \inf  \{ \reg(\signal ) \mid \signal \in \X \wedge \Ao \signal =  \data  \}$. Due to (\ref{a-dist2}),  $(\signal_k)_{k \in \N} $ is contained in  $\{\signal \in \X \mid \dist(\Ao(\signal), \data) + \alpha \reg(\signal) \leq C\}$ for some $C>0$ which   is bounded  according to (\ref{a-c}). By  (\ref{a1}) $\X$ is reflexive and therefore $(\signal_k)_{k \in \N} $ has a weak accumulation  point $\signal^\plus$. From (\ref{a1}), (\ref{a-a}), (\ref{a-reg}) we conclude that  $\signal^\plus$ is an $\reg$-minimizing solution of $\Ao \signal = \data$.
The case of $\reg_n$-minimizing solutions follows analogous.
\end{proof}

\subsection{Convergence}

Next we proof that discrete NETT converges as the noise level goes to zero  and the discretization as well as the regularization parameter are chosen properly. We write $\dom_{n,M} \coloneqq \{\signal \in \dom \cap\X_n \mid  \reg_n(\signal) \leq  M\}$ and formulate the following approximation conditions for obtaining convergence.

\begin{assumption}[Conditions for convergence]\mbox{}\\
 Element $\signal^\plus \in \dom$  satisfies the following for all $M>0$:
\begin{enumerate}[label=(C\arabic*), ref=C\arabic*]
\item\label{c1} $\exists (\zsignal_n)  \in \prod_{n\in \N} (\dom \cap \X_n)$ with  $\lambda_n \coloneqq \abs{\reg_n ( \zsignal_n  ) -  \reg(\signal^\plus)}
            \to 0$.
\item  \label{c2}
    $ \rho_n \coloneqq \sup_{\signal \in \dom_{n,M}}
    \lvert \reg_n(\signal ) - \reg(\signal)\rvert  \to 0 $.

\item\label{c3}
$\gamma_n \coloneqq \dist(\Ao_{n} \zsignal_n, \Ao\signal^\plus  ) \to 0$.

\item \label{c4}
   $ a_n \coloneqq \sup_{\signal \in \dom_{n,M}}
    \lvert \dist(\Ao_n\signal, \Ao\signal^\plus ) - \dist(\Ao\signal, \Ao\signal^\plus )\rvert \to 0 $.
         \end{enumerate}

\end{assumption}

 Conditions  (\ref{c1}) and (\ref{c3}) concerns the approximation of the true unknown $\signal$ with elements in the discretization space, that is compatible with the discretization of
 the forward operator and regularizer.     Conditions   (\ref{c2}) and    (\ref{c4})  are uniform approximation properties of the operator and the regularizer on $\reg_n$-bounded sets.

\begin{theorem}[Convergence]\label{thm:conv}
   Let  (\ref{a1})-(\ref{a9}) hold, $\data \in \Ao(\dom)$ and let $\signal^\plus $ be an $\reg$-minimizing solution of  $\Ao \signal = \data$ that satisfies  (\ref{c1})-(\ref{c4}). Moreover, suppose    $(\delta_k)_{k\in\N} \in (0, \infty)^\N$ converges to zero and
   $(\data_k)_{k\in\N} \in \Y^\N$ satisfies $\dist( \data, \data_k ) \leq \delta_k$.
  Choose $(\alpha_k)_{k \in \N}$ and $(n_k)_{k \in \N}$
  such that as $k \to \infty$ we have
    \begin{align} \label{eq:para1}
    & \alpha_k \to  0
    \\ \label{eq:para2}
      &  n_k \to   \infty
  \\ \label{eq:para3}
      & (\delta_k  + \dist(\Ao_{n_k} \zsignal_{n_k}, \data))/{\alpha_k}
     \to  0    \,.
     \end{align}
Then for   $\signal_k  \in \argmin \tik_{\data_k,\delta_k,n_k}$   the following hold:
  \begin{enumerate}[label=(\alph*)]
  \item \label{conv1} $(\signal_k)_{k\in\N}$ has a weakly convergent subsequence $(\signal_{\sigma(k)})_{k\in\N}$
  \item \label{conv2}  The weak limit of  $(\signal_{\sigma(k)})_{k\in\N}$ is an  $\reg$-minimizing solution of $\Ao \signal = \data$.
  \item \label{conv3} $\reg_{\sigma(k)}(\signal_{\sigma(k)}) \to \reg(\signal^\star)$, where $\signal^\star$ is the weak limit of  $(\signal_{\sigma(k)})_{k\in\N}$.
   \item \label{conv4}  If the   $\reg$-minimizing solution of $\Ao \signal = \data$  is unique, then
   $(\signal_k)_{k\in\N} \rightharpoonup \signal^\plus$.
  \end{enumerate}
\end{theorem}

\begin{proof}
    For convenience and some abuse of notation we use the abbreviations $\reg_{k} \coloneqq \reg_{n_k}$, $\Ao_k  \coloneqq  \Ao_{n_k}$, $a_k  \coloneqq  a_{n_k}$, $\zsignal_k \coloneqq  \zsignal_{n_k}$ and $\rho_k \coloneqq  \rho_{n_k}$. Because $\signal_k$ is a minimizer of  the discrete NETT functional $\tik_{\data_k,\delta_k,n_k}$ by (\ref{a-dist}) we have
\begin{multline*}
    \dist(\Ao_k \signal_k, \data_k) + \alpha_k\reg_k (\signal_k) \leq \dist(\Ao_k \zsignal_k, \data_k) + \alpha_k\reg_{k}(\zsignal_k) \\
    \leq \tau\dist(\Ao_k \zsignal_k, \data) + \tau \dist(\data, \data_k) + \alpha_k\reg_{k}(\zsignal_k)
    =
    \tau \dist(\Ao_k \zsignal_k, \data) + \tau \delta_k + \alpha_k\reg_{k}(\zsignal_k)
\end{multline*}
According to (\ref{c1}), \eqref{eq:para1}, we get
\begin{align} \label{eq:conv1}
       \dist(\Ao_k \signal_k, \data_k)
       & \leq  \tau ( \dist(\Ao_k \zsignal_k, \data) +  \delta_k ) \,,
\\ \label{eq:conv2}
        \reg_k(\signal_k)
        &\leq
        \tau \cdot  \frac{  \dist(\Ao_k \zsignal_k, \data_k) +  \delta_k}{\alpha_k} +\reg_{k}(\zsignal_k) \,.
\end{align}
According to (\ref{c1}), (\ref{c3}), \eqref{eq:para2}, \eqref{eq:para3} the right hand side in \eqref{eq:conv1} converges  to zero and the right hand side in \eqref{eq:conv2}   to $\reg(\signal_\plus)$.   Together with (\ref{c2}) we obtain  $\reg_k(\signal) \leq  \reg_k(\signal_k) + \rho_k \to \reg(\signal_\plus)$ and $\dist(\Ao \signal_k, \data)  \leq  \tau \dist(\Ao_k \signal_k, \data_k) + \tau \delta_k \leq \tau  \dist(\Ao \signal_k, \data) + \tau a_k + \tau \delta_k \to 0$.  This shows that $(\dist(\Ao \signal_k, \data) + \reg(\signal_k))_{k \in\N }$ is bounded and by (\ref{a1}), (\ref{a9}) there exists a weakly convergent  subsequence $(\signal_{\sigma(k)})_{k\in\N}$. We denote the  weak limit by  $\signal^\star\in \X$. From  (\ref{a-dist}), (\ref{a-a}) we obtain $ \Ao \signal= \data$.
The weak  lower semi-continuity of $\reg$ assumed in  (\ref{a-reg}) shows
\begin{multline*}
    \reg(\signal^\star) \leq \liminf_{k} \reg(\signal_{\sigma(k)}) \leq \limsup_k \reg(\signal_{\sigma(k)}) \\ \leq \limsup_{k} (\reg_{\sigma(k)}(\signal_{\sigma(k)}) + \rho_k ) \leq \reg(\signal^\plus) \,.
\end{multline*}
Consequently, $\signal^\star$ is an $\reg$-minimizing solution of $\Ao \signal = \data$ and $\reg(\signal_{\sigma(k)}) \to  \reg(\signal^\star)$.
If the $\reg$-minimizing solution is unique then $\signal^\plus$ is the only weak accumulation point of $(\signal_k)_{k\in\N}$ which concludes the proof.
\end{proof}

\subsection{Convergence rates}

Next we derive quantitative error estimates (convergence rates) in terms of the  absolute Bregman distance. Recall that a function $\reg \colon \X \to [0, \infty]$ is G\^{a}teaux differentiable at some $\signal^\star \in \X$ if the directional derivative $\reg^\prime(\signal^\star )(h) \coloneqq (\reg(\signal^\star +  t h) - \reg   (\signal^\star) )/t$ exist for every $h \in \X$. We denote by  $\reg^\prime(\signal^\star )$ the G\^{a}teaux derivative of $\reg$ at $\signal$. In  \cite{li2020nett} we introduced the absolute Bregman distance  $\Breg_\reg(\edot,\signal^\star )\colon \X \to [0,\infty]$ of a G\^{a}teaux differentiable functional $\reg \colon \X \to [0, \infty]$   at $\signal^\star \in \X$ with respect to $\reg$  defined by
\begin{equation}\label{eq:absbreg}
    \forall \signal\in \X \colon \Breg_\reg(\signal, \signal^\star ) \coloneqq \abs{\reg(\signal) - \reg(\signal^\star ) - \reg^\prime(\signal^\star ) (\signal - \signal^\star ) }\,.
\end{equation}

We write $\sup_{\data^\delta} H(\data^\delta) \coloneqq \sup\{ H(\data^\delta) \mid \data^\delta \in \X \wedge \dist(\Ao\signal^\plus,\data^\delta) \leq \delta \}$.   Convergence rates in terms of the Bregman distance are derived under a smoothness assumption on the true solution in the form of a certina variational inequality. More precisely we assume the following:

\begin{assumption}[Conditions for convergence rates]\mbox{}\\
 Element $\signal^\plus \in \dom$  satisfies the following for all $M, \delta >0$:
\begin{enumerate}[label=(R\arabic*), ref=R\arabic*]
        \item\label{r1}
        Items (\ref{c1}), (\ref{c2}) hold.
       \item\label{r2}
       $\gamma_{n,\delta} \coloneqq \sup_{\data^\delta}
        \abs{\dist(\Ao_{n} \zsignal_n, \data^\delta  )  -
        \dist(\Ao\signal^\plus, \data^\delta  )}\to 0$.
        \item\label{r3}
         $ a_{n,\delta} \coloneqq
         \sup_{\data^\delta} \sup_{\signal \in \dom_{n,M}}
    \lvert \dist(\Ao_n\signal, \data^\delta ) - \dist(\Ao\signal, \data^\delta )\rvert \to 0 $.

\item\label{r4}  $\reg$ is G\^{a}teaux differentiable at $\signal^\plus$
\item\label{r5}  There  exist  a concave, continuous, strictly increasing $\phi\colon [0,\infty)\to[0,\infty)$ with $\phi(0)=0$ and  $\epsilon, \beta>0$ such that for all $\signal\in \X$
\begin{equation*}
    \abs{\reg(\signal)-\reg(\signal^\plus)} \leq \epsilon
    \Rightarrow
    \beta\Breg_\reg(\signal,\signal^\plus) \leq \reg(\signal) - \reg(\signal^\plus) + \phi\bigl(\dist(\Ao \signal, \Ao \signal^\plus)\bigr) \,.
\end{equation*}
\end{enumerate}
\end{assumption}

According to (\ref{r5}) the inverse function $\phi^{-1}\colon [0,\infty)\to [0,\infty)$ exists and  is convex. We denote by $\phi^{-\ast}(s) \coloneqq \sup \{ s \, t - \phi^{-1}(t) \mid t  \geq 0 \} $ its  Fenchel conjugate.

\begin{proposition}[Error estimates]\label{prop:rate1}
Let  $\data\in \Ao(\dom)$ and $\signal^\plus$ be an $\reg$-minimizing solution of $\Ao \signal = \data$ such that  (\ref{a1})-(\ref{a9}) and (\ref{r1})-(\ref{r5})  are satisfied.  For   $\data^\delta \in \Y$ with $\dist(\data,\data^\delta) \leq \delta$ let  $\signal_{\alpha,n}^{\delta} \in \argmin \tik_{\data^\delta,\alpha,n}$. Then for sufficient small $\delta , \alpha >0$ and sufficiently large $n \in \N$, we have the error estimate
\begin{equation}\label{eq:rate}
    \Breg_\reg(\signal_{\alpha,n}^{\delta}, \signal^\plus) \leq   \frac{a_{n, \delta} + \gamma_{n,\delta} + \delta}{\alpha} +   \rho_n + \lambda_n   +   \phi(\tau\delta) + \frac{\phi^{-\ast}(\tau\alpha)}{\tau \alpha}   \,.
 \end{equation}	
\end{proposition}

\begin{proof}
We have
  $  \dist(\Ao_n\signal_{\alpha,n}^{\delta}, \data^\delta) + \alpha\reg_n(\signal_{\alpha,n}^{\delta}) \leq \dist(\Ao_n \zsignal_n, \data^\delta) + \alpha \reg_n(\zsignal_n) $. According to Theorem~\ref{thm:conv} we can assume  $\abs{\reg(\signal_{\alpha,n}^{\delta})-\reg(\signal^\plus)} < \epsilon$ and with   (\ref{r5}) we obtain
\begin{align*}
    &\alpha \beta \Breg_\reg(\signal_{\alpha,n}^{\delta}, \signal^\plus)
       \\ & \leq
    \alpha\reg(\signal_{\alpha,n}^{\delta}) - \alpha \reg(\signal^\plus) + \alpha\phi(\dist(\Ao \signal_{\alpha,n}^{\delta}, \data))
    \\ & \leq
    \alpha \reg_n(\signal_{\alpha,n}^{\delta}) - \alpha \reg(\zsignal_n)
    + \alpha\rho_n + \alpha\lambda_n
    + \alpha\phi(\dist(\Ao \signal_{\alpha,n}^{\delta}, \data ))
    \\ &  \leq
     \dist(\Ao_n \zsignal_n, \data^\delta) - \dist(\Ao_n\signal_{\alpha,n}^{\delta}, \data^\delta)
    + \alpha\rho_n + \alpha\lambda_n
    + \alpha\phi(\dist(\Ao \signal_{\alpha,n}^{\delta}, \data ))
    \\ &  \leq
     \delta - \dist(\Ao\signal_{\alpha,n}^{\delta}, \data^\delta)
    + \gamma_{n,\delta} + a_{n,\delta}
    + \alpha\rho_n + \alpha\lambda_n
    + \alpha\phi(\tau\delta) + \alpha\phi(\tau\dist(\Ao \signal_{\alpha,n}^{\delta}, \data^\delta ))
    \\ &  \leq
     \delta
    + \gamma_{n,\delta} + a_{n,\delta}
    + \alpha\rho_n + \alpha\lambda_n
    + \alpha\phi(\tau\delta) + \tau^{-1} \phi^{-*}(\tau\delta ) \,.
\end{align*}
where we used Young's inequality $\alpha\phi(\tau t) \leq t + \tau^{-1}\phi^{-\ast}(\tau\alpha)$ for the last step.
\end{proof}

\begin{remark}
The error estimate \eqref{eq:rate} includes  the approximation quality of the discrete or inexact forward  operator $\Ao_n$ and the discrete or inexact regularizer  $\reg_n$ described by  $a_{n,\delta}$ and $\rho_n$,  respectively.
What might be unexpected at first is the inclusion of two new parameters $\lambda_n$ and $\gamma_{n,\delta}$. These factors both arise from the approximation of $\X$ by the finite dimensional spaces $\X_n$, where $\gamma_{n,\delta}$ reflects approximation accuracy in the image of the operator $\Ao$ and $\lambda_n$ approximation accuracy with respect to the true regularization functional $\reg$.   Note that in the case where the forward  operator, the  regularizer and the solution space $\X$ are given precisely,  we have $a_{n,\delta} = \gamma_{n,\delta} = \lambda_n =  \rho_n = 0$. In this particular case we recover the estimate derived for the NETT in \cite{li2020nett}.
\end{remark}

\begin{theorem}[Convergence rates]\label{thm:rate1}
Let the assumptions of Proposition~\ref{prop:rate1}  hold and consider the parameter choice rule $\alpha(\delta) \asymp \delta/\phi(\delta)$ and let the approximation errors satisfy $a_{n, \delta} + \gamma_{n,\delta} = \mathcal O(\delta)$, $\rho_n + \lambda_n = \mathcal O(\phi(\tau \delta))$. Then we have the convergence rate
\begin{equation}\label{eq:rate-a}
	 \Breg_\reg(\signal_{\alpha(\delta),n(\delta)}^{\delta}, \signal^\plus)
	 = \mathcal O(\phi(\tau \delta)) \,.
\end{equation}
\end{theorem}

\begin{proof}
Noting that  $\phi^{-*} (\tau \delta / \phi(\tau \delta))/\delta$ remains bounded as  $\delta \to 0$, this directly follows from  Proposition~\ref{prop:rate1}
\end{proof}

Next we verify that  a variational inequality of the form (\ref{r5}) is satisfied with $\phi(t) = c \sqrt{t}$ under a typical source like condition.

\begin{lemma}[Variational inequality under  source condition]
\label{lem:source}
Let  $\reg$, $\Ao$ be  G\^{a}teaux differentiable at $\signal^\plus \in \X$, consider the distance measure $\dist(\data_1, \data_2 ) = \| \data_1 - \data_2  \|^2$  and assume there exist $\eta \in \X^\star$ and $c_1,c_2,\epsilon>0$ with $c_1\norm{\eta}<1$ such that for all $\signal\in \X$ with $\abs{\reg(\signal)-\reg(\signal^\plus)}\leq \epsilon$ we have
\begin{equation}\label{eq:porp:conv_as}
    \begin{aligned}
       & \reg^\prime(\signal^\plus) = \Ao^\prime(\signal^\plus)^\ast\eta   \\
      &\norm{\Ao \signal - \Ao \signal^\plus - \Ao^\prime(\signal^\plus)(\signal-\signal^\plus)} \leq c_1\Breg_\reg(\signal,\signal^\plus)
      \\
       & \reg(\signal^\plus) - \reg(\signal) \leq c_2 \norm{\Ao \signal - \Ao \signal^\plus}
       \,.
    \end{aligned}
\end{equation}
Then (\ref{r5}) holds with    $\phi(t) = (\norm{\eta}+2c_2)\sqrt{t}$ and $\beta=1-c_1\norm{\eta}$.
\end{lemma}

\begin{proof}
Let $\signal\in \X$ with $\abs{\reg(\signal)-\reg(\signal^\plus)}\leq \epsilon$. Using the Cauchy-Schwarz inequality and equation~\eqref{eq:porp:conv_as}, we can estimate
\begin{align*}
    \abs{\innerprod{\reg^\prime(\signal^\plus), \signal-\signal^\plus}} \leq& \norm{\Ao^\prime(\signal^\plus) (\signal-\signal^\plus)}\norm{\eta}\\ \leq& \norm{\Ao \signal - \Ao \signal^\plus}\norm{\eta} + \norm{\Ao \signal - \Ao \signal^\plus - \Ao^\prime(\signal^\plus)(\signal-\signal^\plus)}\norm{\eta} \\
    \leq& \norm{\Ao \signal - \Ao \signal^\plus}\norm{\eta} +  c_1\norm{\eta}\Breg_\reg(\signal,\signal^\plus)\,.
\end{align*}
Additionally, if $\reg(\signal)\geq\reg(\signal^\plus)$, we have $\abs{\reg(\signal)-\reg(\signal^\plus)}=\reg(\signal) - \reg(\signal^\plus)$, and on the other hand if $\reg(\signal)<\reg(\signal^\plus)$, we have $\abs{\reg(\signal)-\reg(\signal^\plus)}\leq \reg(\signal)-\reg(\signal^\plus) + 2(\reg(\signal^\plus)-\reg(\signal)) \leq \reg(\signal)-\reg(\signal^\plus) + 2c_2 \norm{\Ao \signal - \Ao \signal^\plus}$.
Putting this together we get
\begin{multline*}
    \Breg_\reg(\signal,\signal^\plus) \leq \abs{\reg(\signal)-\reg(\signal^\plus)} + \abs{\innerprod{\reg^\prime(\signal^\plus), \signal-\signal^\plus}}\\
    \leq \reg(\signal)-\reg(\signal^\plus) + (\norm{\eta}+2c_2)\norm{\Ao \signal -\Ao \signal^\plus} + c_1\norm{\eta}\Breg_\reg(\signal,\signal^\plus) \,,
    \end{multline*}
and thus  $(1-c_1\norm{\eta})\Breg_\reg(\signal,\signal^\plus)\leq \reg(\signal)-\reg(\signal^\plus) + (\norm{\eta}+2c_2)\norm{\Ao \signal - \Ao \signal^\plus}$.
\end{proof}

\begin{corollary}[Convergence rates under source condition]\label{cor:rates}
    Let the conditions of Lemma~\ref{lem:source} hold  and suppose
    \begin{align*}
    &\alpha(\delta) \asymp \sqrt{\delta}
    \\
    &\abs{ \reg_{n(\delta)}( \zsignal_{n(\delta)} ) - \reg(\signal^\plus )}  = \mathcal O(\sqrt \delta)
    \\
    & \sup \{ \lvert \reg_{n(\delta)}(\signal ) - \reg(\signal)\rvert \mid  \signal \in \dom_{{n(\delta)},M} \}= \mathcal O(\sqrt \delta)
    \\
    &\norm{ \Ao_{n(\delta)} \zsignal_{n(\delta)} - \Ao\signal^\plus } = \mathcal O(\sqrt \delta)
    \\
    & \sup \{
       \norm{\Ao_{n(\delta)} \signal - \Ao \signal } \mid \signal \in \dom_{{n(\delta)},M} \}
       = \mathcal O(\sqrt \delta)
    \\
    & \sup \{
       \norm{\Ao_{n(\delta)} \signal  } \mid \signal \in \dom_{{n(\delta)},M} \} < \infty \,.
    \end{align*}
    Then  we have the convergence rates result
    \begin{equation}\label{eq:convergence_rate}
    	 \Breg_\reg(\signal_{\alpha(\delta),n(\delta)}^{\delta}, \signal^\plus)
	 = \mathcal{O}(\sqrt{\delta})\,.
    \end{equation}
\end{corollary}

\begin{proof}
    Follows from Theorem~\ref{thm:rate1} and Lemma~\ref{lem:source}. Note that we use $\| \edot \|$ in the theorem, while $\dist (\data_1,\data_1) = \| \data_1-\data_2\|^2$ uses the  squared norm $\| \edot \|^2$ and thus the approximation rates for the terms concerning $\Ao_{n(\delta)}$ are order $\sqrt{\delta}$ instead of $\delta$ as in Theorem~\ref{thm:rate1}.
\end{proof}

In Corollary~\ref{cor:rates}, the approximation quality of the discrete operator $\Ao_n$ and the discrete and inexact regularization functional $\reg_n$ need to be of the same order.

\section{Application to a limited data problem in PAT}
\label{sec:PAT}

Photoacoustic Tomography (PAT) is an emerging non-invasive coupled-physics biomedical imaging technique with high contrast and high spatial resolution~\cite{kruger1995photoacoustic,paltauf2007photoacoustic}.
It works by illuminating a semi-transparent sample with short optical
pulses which causes heating of the sample followed by expansion and the subsequent emission of an acoustic wave.
Sensors on the outside of the sample measure the acoustic wave and these measurements are then used to reconstruct the initial pressure
$f\colon \R^d \to \R$, which provides information about the interior of the object.
The cases $d=2$ and $d=3$ are relevant for applications in PAT.
Here we only consider the case $d=2$ and assume a circular measurement geometry. The 2D case arises for
example when using integrating line detectors in PAT~\cite{paltauf2007photoacoustic}.

\subsection{Discrete forward operator}

The pressure data $p \colon\R^2\times [0,\infty)\to \R$ satisfies the wave equation $(\partial^2_t    - \Delta) p(\rr,t) = 0 \text{ for } (\rr,t) \in \R^2 \times (0,\infty)$ with initial data $p(\edot,0) = $ and $\partial_t p(\edot,0) = 0$.
In the case of circular measurement geometry one assumes that  $f$ vanishes outside the unit disc $D_1  \coloneqq \{\rr\in\R^2 \mid \|\rr\| < 1\}$ and the measurement sensors are located on the boundary $\partial D_1 = \mathbb{S}^1$.
We assume that the phantom will not generate any data for some region $I \subseteq D_1$, for example when the acoustic pressure generated inside $I$ is too small to be recorded. This masked PAT problem consists in the recovery of the function $f$ from sampled noisy measurements of $g = \Wo ( \mathds{1}_{I^c} f )$
where $\Wo$ denotes the solution operator of the wave equation  and $\mathds{1}_{I^c}$ the indicator function on $I^c \coloneqq \R^2\setminus I$.
Note that the resulting inverse problem can be seen of the combination of an inpainting problem and in inverse problems for the wave equation.

\begin{figure}[htb!]
    \includegraphics[width=\textwidth]{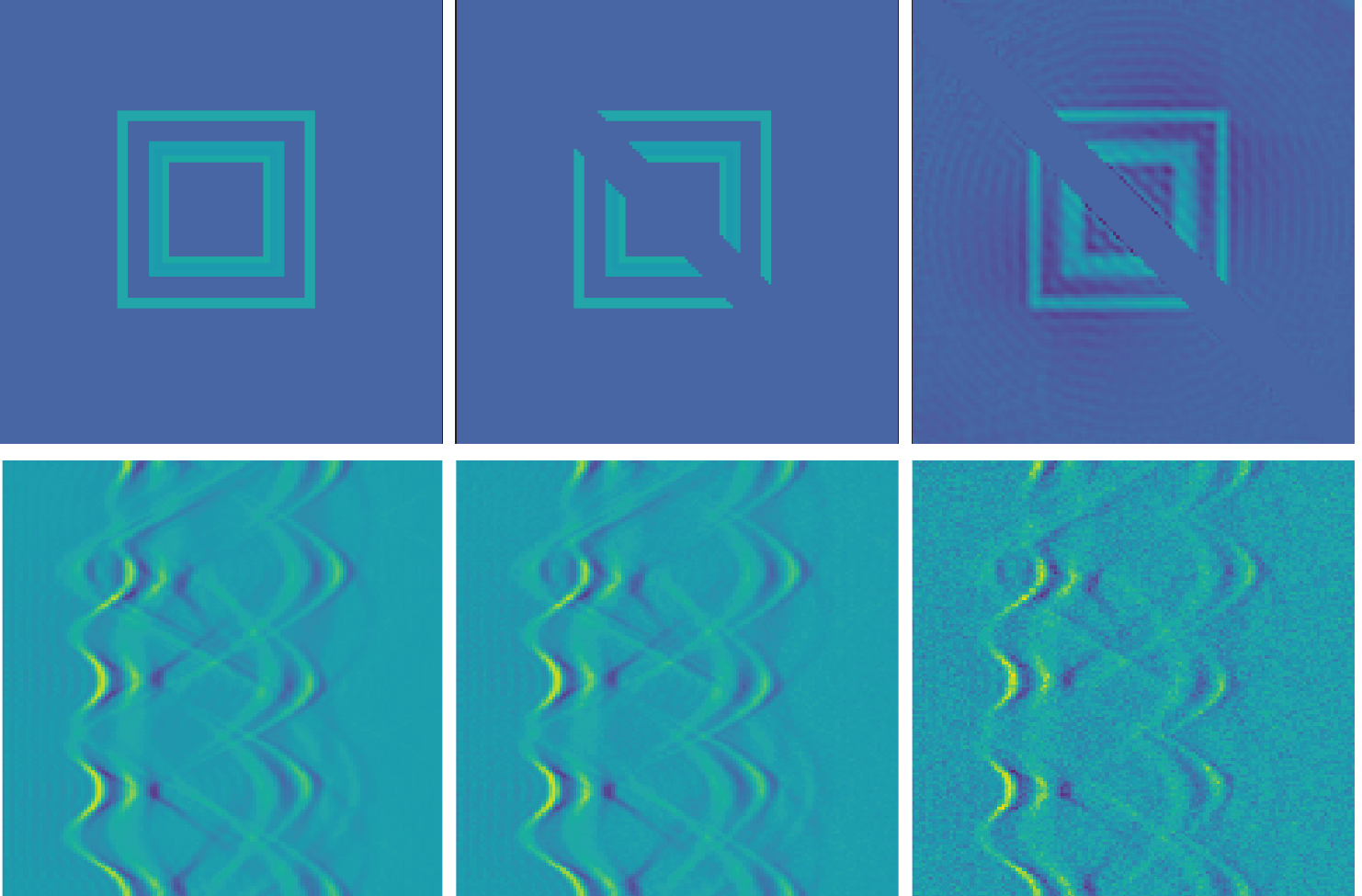}
    \caption{Top from left to right: phantom, masked phantom and initial reconstruction $\Ao^\plus \Ao \signal$.
    Bottom from left to right: data without  noise,  low noise $\sigma=0.01$  and high noise $\sigma=0.1$.}
    \label{fig:data}
\end{figure}

In order to implement the PAT forward operator we use a basis ansatz $f(\rr) = \sum_{i=1}^{N\times N} \signal_i \psi(\rr-\rr_i)$ where $\signal_i \in\R$ are basis coefficients and $\psi \colon \R^2 \to \R$ a generalized Kaiser-Bessel (KB) and $\rr_i = (i-1)/N$ with $ i = (i_1, i_2) \in \{1, \dots, N\}^2$.
The generalized KB  functions are  popular in tomographic inverse problems~\cite{matej1996practical,schwab2018galerkin,wang2014discrete,wang2012investigation}  and denote radially symmetric functions with support in $D_R$ defined by  
\begin{equation}
    \psi(\rr) \coloneqq
        \bigl( 1-\|\rr\|^2/R^2 \bigr)^{m /2}  \, \frac{I_m\bigl(\gamma\sqrt{1-\|\rr\|^2/R^2} \bigr)}{I_m(\gamma)}
        \quad \text{for} \|\rr\|\leq R \,.
\end{equation}
Here $I_m$ is the modified Bessel function of the first kind of order $n \in\N$ and the parameters $\gamma>0$ and $R$ denote the window taper and  support radius, respectively. Since $\Wo$ is linear we have $\Wo f = \sum_{i=1}^{N \times N} \signal_i \Wo (\psi(\edot-\rr_i)) $.  For convenience we will use a pseudo-3D approach where use the 3D solution of $\Wo \psi$ for which there exists an analytical representation~\cite{wang2014discrete}.
Denote by $\sr_k$ uniformly spaced sensor locations on $\mathbb{S}^1$ and by $t_j > 0$ uniformly sampled measurement times in $[0,2]$.
Define the $N_t N_s \times N^2$ model matrix by $\Wd_{N_t(k-1)+j,N(i_1-1)+i_2} =  \Wo (\psi(\edot-\rr_i))(\sr_k,t_j)$ and an $N^2\times N^2$ diagonal matrix by  $(\Mo_{I})_{N(i_1-1)+i_2,N(i_1-1)+i_2} = 1$ if  $\rr_i \in I^c$ and zero otherwise. Let $\Wd \Mo_{I} = \Uo \Sigmao \transpose{\Vo}$ be the singular valued decomposition. We then consider the discrete forward matrix $\Ao  = \Uo  \Sigmao_\star  \transpose{\Vo}$
where $\Sigmao_\star$ is the diagonal matrix derived from $\Sigmao$ by setting singular values smaller than some $\sigma_\star$ to zero. In our experiments we use $N = N_t=128$, $N_s=150$ and take $I$ fixed as a diagonal stripe of width $0.34$.

\begin{algorithm}
    \DontPrintSemicolon
    \KwIn{$\data \in \Y$, $\signal_0\in \X$, $\alpha,s>0$.}
    \KwOut{reconstruction $\signal$}
    \For{$\ell = 1, \ldots, N_{\rm iter}$}{
        $\signal_{\ell+1/2} = \signal_{\ell-1} - s\alpha\nabla \reg_n(\signal_{\ell-1})$\\
        $\signal_\ell = (\transpose{\Ao}\Ao - s \operatorname{Id})^{-1}\left(\transpose{\Ao}\data + s \signal_{\ell+1/2}\right)$
    }
    \caption{NETT optimization.
    }\label{alg:NETT}
\end{algorithm}

\subsection{Discrete NETT}
\label{sec:NETT_training}

We consider the discrete NETT  with discrepancy term $ \dist(\Ao\signal, \data^\delta) = \lVert\Ao \signal - \data^\delta \rVert_2^2/2$ and regularizer given by
\begin{equation}\label{eq:concrete_NETT}
         \reg^{(m)} (\signal) =  \norm{\signal - \net^{(m)}(\signal)}_2^2 +   \beta \norm{\nabla \signal}_{1,\epsilon} \,,
\end{equation}
where $ \norm{\nabla \signal}_{1,\epsilon} \coloneqq \sum_{i_1,i_2=1}^{128} \sqrt{ \lvert \rr_{i_1+1,i_2}-\rr_i\rvert^2 + \lvert\rr_{i_1,i_2+1}-\rr_{i_1,i_2} \rvert^2 + \epsilon^2}$   with $\epsilon > 0$  is a smooth version of the total variation ~\cite{acar1994analysis}and $\net^{(m)}$ is a learnable network. We take  $\net^{(m)}$ as the U-Net~\cite{ronneberger2015unet} with residual connection, which has first been applied to PAT image reconstruction in~\cite{antholzer2019deep}.
We generate training data that consist of square shaped rings with random profile and random location. See Figure~\ref{fig:data} for an example of one such phantom (note that all plots in signal space use the same colorbar) and the corresponding data. We get a set of phantoms $\signal_1, \dots , \signal_{1000}$ and corresponding basic reconstructions $\bad_a \coloneqq \Ao^\plus (\Ao \signal_a + \eta_a)$, where $\Ao^\plus$ is the pseudo-inverse and $\eta_a$ is Gaussian noise with standard deviation of $\sigma\norm{\Ao \signal_a}_\infty$ with $\sigma=0.01$. The networks are trained by minimizing
$ \sum_{a=1}^{1000} \lVert \net^{(m)} (\bad_a) - \signal_a \rVert_1 + \gamma \lVert \net^{(m)} (\signal_a) - \signal_a\rVert_1 $ where we used the Adam optimizer with learning rate 0.01 and $\gamma=0.1$. The considered loss is that we want the trained regularizer to give small values for $\signal_a$ and large values for $\bad_a$. The strategy is similar to~\cite{li2020nett} but we use the final output of the network for the regularizer as   proposed in~\cite{antholzer2019nett}. To minimize~\eqref{eq:concrete_NETT} we use Algorithm~\ref{alg:NETT}  which implements a forward-backward scheme  \cite{combettes2011proximal}.

\begin{figure}[htb!]
    \centering
    \includegraphics[width=\textwidth]{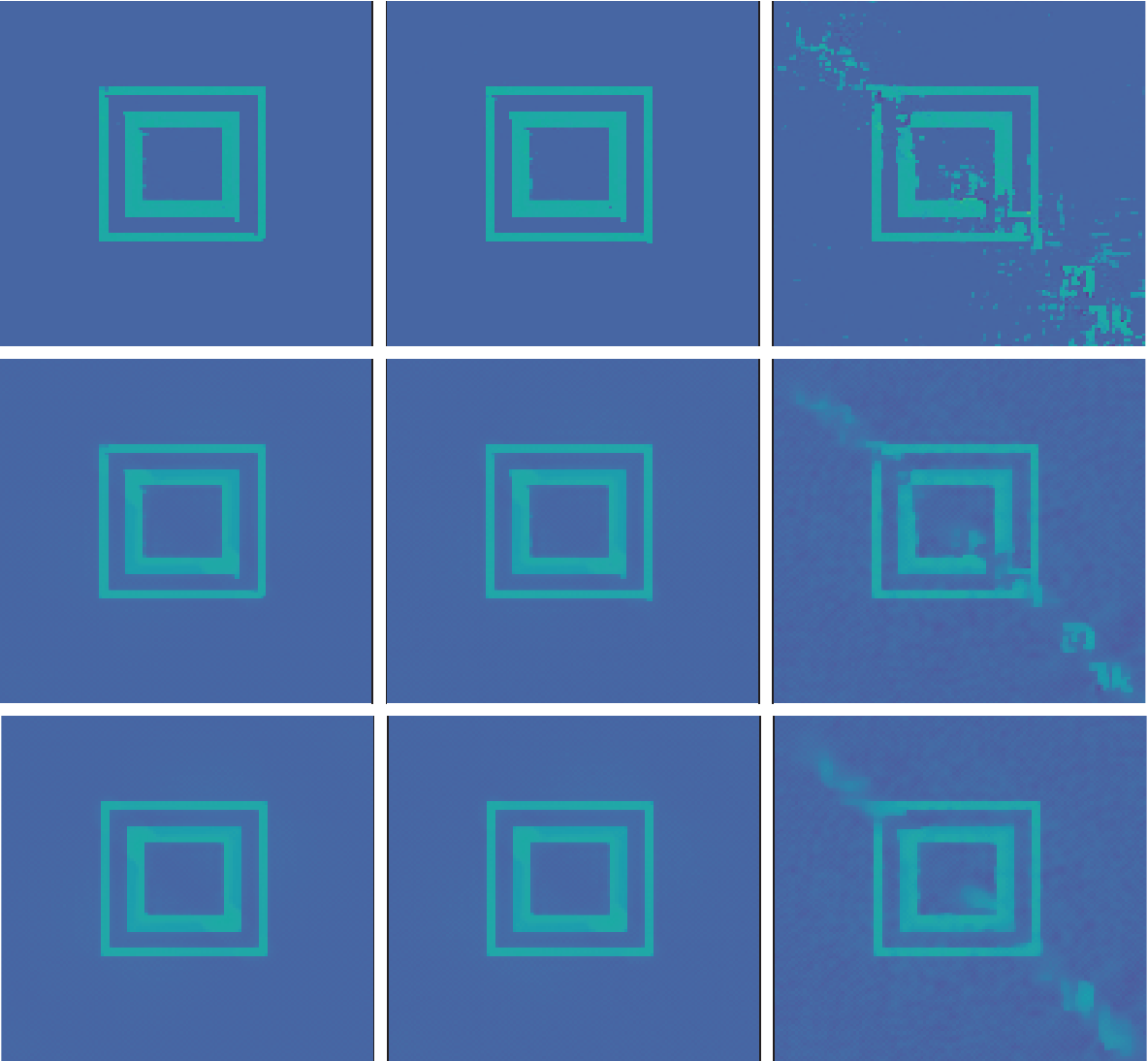}
    \caption{Top row: reconstructions using post-processing network  $\Phi^{(1)}$.
    Middle row: NETT reconstructions using   $\reg^{(1)}$.
    Bottom row: NETT reconstructions using $\reg^{(3)}$.
    From Left to Right: Reconstructions from data without noise, low noise ($\sigma=0.01$) and high noise ($\sigma=0.1)$.}
    \label{fig:rec1}
\end{figure}

\begin{figure}[htb!]
    \centering
    \includegraphics[width=0.8\textwidth]{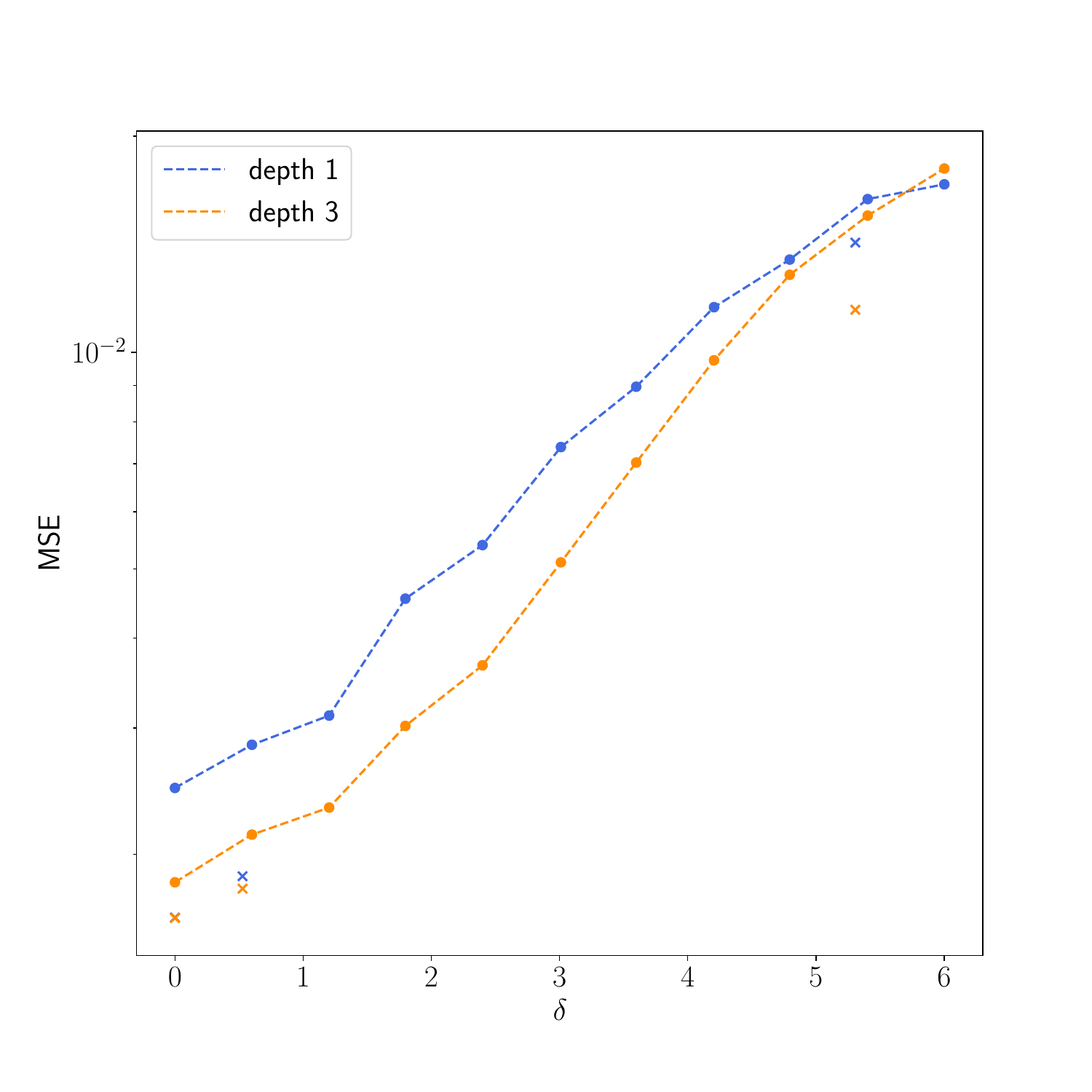}
    \caption{Semilogarithmic plot of the mean squared errors of the NETT  using $\reg^{(1)}$ and $\reg^{(3)}$ depending on the noise level. The crosses are the values for the phantoms in Figure~\ref{fig:rec1}.}
    \label{fig:convergence}
\end{figure}

\subsection{Numerical results}

For the  numerical results we train two regularizers $\reg^{(1)}$ and $\reg^{(3)}$ as described in Section~\ref{sec:NETT_training}. The networks are implemented using PyTorch~\cite{pytorch}. We also use PyTorch in order to calculate the gradient $\nabla_x \reg^{(m)}$. We take $N_{\rm iter} = 15$, $s=0.25$ and $\signal_0 = \net^{(m)} \transpose{\Ao}\data$ in Algorithm~\ref{alg:NETT} and compute  the inverse $(\transpose{\Ao}\Ao - s \operatorname{Id})^{-1}$ only once and then use it for all examples. We set $\alpha =0.015$ for the noise-free case, $\alpha = 0.016$ for the low noise case and $\alpha=0.02$ for the high noise cases, respectively, and selected a fixed $\beta = 15$. We expect that the NETT functional will yield better results due to data consistency, which is mainly helpful outside the masked center diagonal.

\begin{figure}[htb!]
    \centering
    \includegraphics[width=\textwidth]{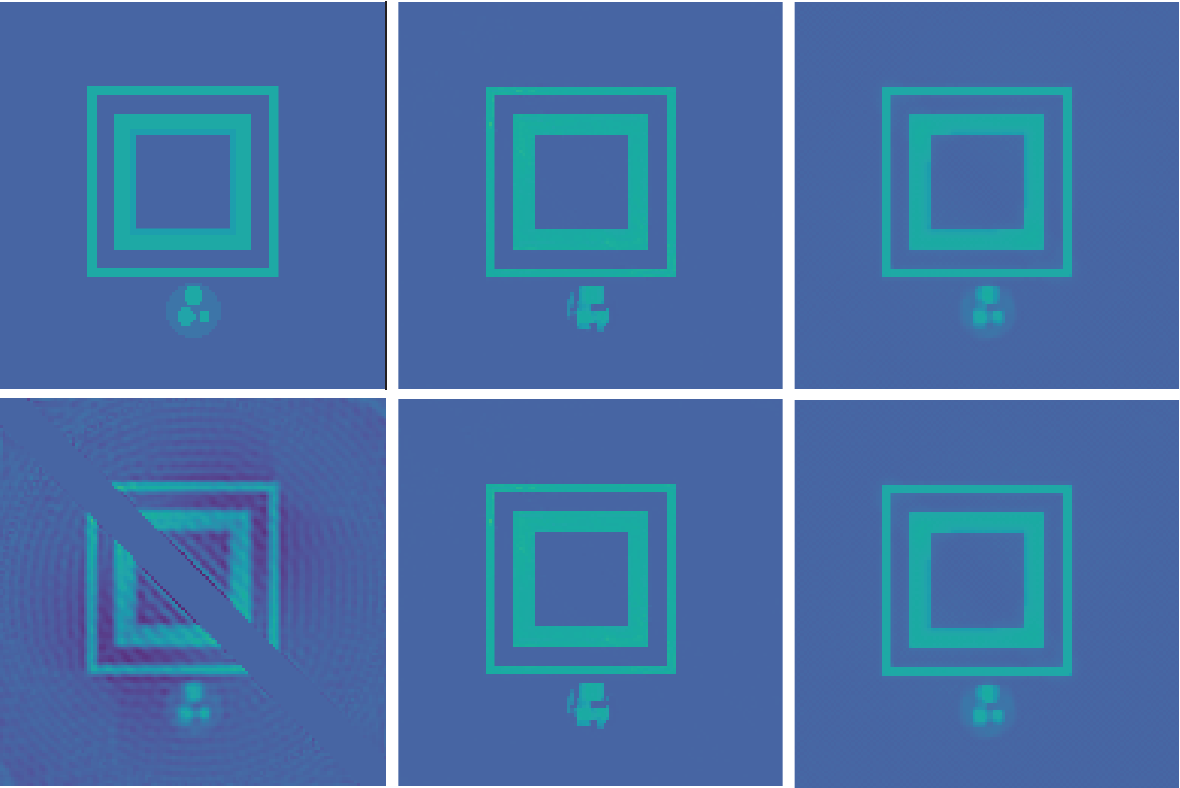}
    \caption{Left column:  phantom with a structure not contained in the training data (top) and  pseudo inverse  reconstruction (bottom).
    Middle column: Post-processing reconstructions using exact (top) and noisy data (bottom).
    Right column: NETT reconstructions using exact (top) and noisy data (bottom).}
    \label{fig:rec2}
\end{figure}

First we use the phantom from the testdata shown in Figure~\ref{fig:data}. The results using post processing and NETT are shown in Figure~\ref{fig:rec1}. One sees that all results with higher noise than used during training are not very good. This indicates that one should use similar noise as in the later applications even for the NETT. Figure \ref{fig:convergence} shows the average  error using 10 test phantoms similar to the on in Figure~\ref{fig:data}. Careful numerical  comparison of the numerical convergence rates and the theoretical results of Theorem~\ref{cor:rates} is an interesting aspect of further research. To investigate the stability of our method with respect to phantoms that are different from the training data we create a phantom with different structures  as seen in Figure~\ref{fig:rec2}. As expected, the  post processing network $\Phi^{(3)}$  is not really able to reconstruct the circles object, since it is quite different from the training data, but it also does not break down completely. On the other hand, the NETT approach yields good results due to data consistency.

\section{Conclusion}
\label{sec:conclusion}

We have analyzed the convergence a discretized NETT approach and derived the convergence rates under certain assumptions on the approximation quality of the involved operators. We  performed numerical experiments using a limited data problem for PAT that  is the combination of an inverse problem for the wave equation and an  inpainting problem. To the best of our knowledge this is the first  such problem studied with  deep learning. The NETT approach yields better results that post processing for phantoms different from the training data. NETT still fails to  recover some missing parts of the phantom in cases  the data contains more noise than the training data. This highlights  the relevance  of using different regularizers for different noise levels.

\section*{Acknowledgments}
S.A. and M.H.  acknowledge support of the Austrian Science Fund (FWF), project P 30747-N32.

\end{document}